\documentclass[graybox]{svmult}

\usepackage{type1cm}        
\usepackage{makeidx}         
\usepackage{graphicx}        
\usepackage{multicol}        
\usepackage[bottom]{footmisc}

\usepackage{newtxtext}       %
\usepackage[varvw]{newtxmath}       

\makeindex             

\usepackage{mathtools}

\usepackage{import}
\usepackage[unicode, bookmarks, colorlinks, breaklinks, pagebackref]{hyperref}
\hypersetup{linkcolor=blue,citecolor=blue,filecolor=black,urlcolor=magenta}

\usepackage[capitalize,nameinlink,noabbrev]{cleveref}

\usepackage{cite}

\newcommand{\tscale}{\xrightharpoonup[]{~2~}}

\usepackage{enumitem}

\newcommand{\ZZ}{{\mathbb Z}}

\newcommand{\RR}{{\mathbb R}}

\newcommand{\DD}{{\mathbb{D}}}
\newcommand{\II}{{\mathbf I}}
\newcommand{\bfa}{\mathbf{a}}
\newcommand{\bfb}{\mathbf{b}}

\newcommand{\cel}{\mathrm{ref}}

\newcommand{\bg}{\mathbf{g}}
\newcommand{\nn}{\mathbf{n}}
\newcommand{\ee}{\mathrm{e}}
\newcommand{\kk}{\mathbf{k}}
\newcommand{\uu}{\mathbf{u}}
\newcommand{\xx}{x}
\newcommand{\yy}{y}

\newcommand{\us}{\mathrm{s}}
\newcommand{\bT}{\boldsymbol \tau}
\renewcommand\vec[1]{\boldsymbol #1} 

\newcommand{\nRe}{{R}_{{e}}}

\newcommand{\nFr}{{F}_{{r}}}

\newcommand{\per}{\mathrm{per}}
\newcommand{\hmeas}{\mathcal{H}^{d-1}}

\newcommand{\bP}{\mathbf{U}} 

\newcommand{\abs}[1]{\left\lvert #1\right\rvert}

\newcommand{\norm}[1]{\left\lVert #1\right\rVert}

\newcommand{\calA}{\mathcal{A}}
\newcommand{\calB}{\mathcal{B}}

\newcommand{\calN}{\mathcal{N}}

\newcommand{\fraM}{\mathfrak{M}}

\DeclareMathOperator{\diam}{diam}

\DeclareMathOperator{\di}{d\!}

\DeclareMathOperator{\Div}{\mathrm{div}}

\DeclareMathOperator{\dist}{\mathrm{dist}}

\def\XXint#1#2#3{{\setbox0=\hbox{$#1{#2#3}{\int}$ }
\vcenter{\hbox{$#2#3$ }}\kern-.6\wd0}}

\usepackage{stix}
\def\dashint{\intbar}

\begin{document}

\title*{Global gradient estimate for a divergence problem and its
  application to the homogenization of a magnetic suspension}
\titlerunning{Global gradient estimate and homogenization of a magnetic suspension}
\author{Thuyen Dang, Yuliya Gorb, and Silvia Jim\'{e}nez  Bola\~{n}os}
\institute{Thuyen Dang \at  University of Houston, Houston, TX 77204 USA, \email{ttdang9@central.uh.edu}
\and  Yuliya Gorb \at National Science Foundation, Alexandria, VA 22314 USA,
\email{ygorb@nsf.gov}
\and  Silvia Jim\'{e}nez  Bola\~{n}os \at Colgate University, Hamilton, NY 13346 USA, \email{sjimenez@colgate.edu}
}
%
%
\maketitle

\abstract*{This paper generalizes the results obtained by the authors
  in \cite{dangHomogenizationNondiluteSuspension2021} concerning the
  homogenization of a non-dilute suspension of magnetic particles in a
  viscous
  flow. 
  More specifically, in this paper, a restrictive assumption on the
  coefficients of the coupled equation, made in
  \cite{dangHomogenizationNondiluteSuspension2021}, that significantly
  narrowed the applicability of the homogenization results obtained,
  is relaxed and a new regularity of the solution of the fine-scale
  problem is proven. In particular, we obtain a global
  $L^{\infty}$-bound for the gradient of the solution of the scalar
  equation
  $-\Div \left[ \bfa \left( x/\varepsilon \right)\nabla
    \varphi^{\varepsilon}(x) \right] = f(x)$, uniform with respect to
  microstructure scale parameter $\varepsilon\ll 1$ in a small
  interval $(0,\varepsilon_0)$, where the coefficient $\bfa$ is only
  \emph{piecewise} H\"{o}lder continuous. Thenceforth, this regularity
  is used in the derivation of the effective response of the given
  suspension discussed in
  \cite{dangHomogenizationNondiluteSuspension2021}.}

\abstract{This paper generalizes the results obtained by the authors
  in \cite{dangHomogenizationNondiluteSuspension2021} concerning the
  homogenization of a non-dilute suspension of magnetic particles in a
  viscous
  flow. 
  More specifically, in this paper, a restrictive assumption on the
  coefficients of the coupled equation, made in
  \cite{dangHomogenizationNondiluteSuspension2021}, that significantly
  narrowed the applicability of the homogenization results obtained,
  is relaxed and a new regularity of the solution of the fine-scale
  problem is proven. In particular, we obtain a global
  $L^{\infty}$-bound for the gradient of the solution of the scalar
  equation
  $-\Div \left[ \bfa \left( x/\varepsilon \right)\nabla
    \varphi^{\varepsilon}(x) \right] = f(x)$, uniform with respect to
  microstructure scale parameter $\varepsilon\ll 1$ in a small
  interval $(0,\varepsilon_0)$, where the coefficient $\bfa$ is only
  \emph{piecewise} H\"{o}lder continuous. Thenceforth, this regularity
  is used in the derivation of the effective response of the given
  suspension discussed in
  \cite{dangHomogenizationNondiluteSuspension2021}.}

\section{Introduction}
\label{sec:introduction}
The purpose of this paper is to generalize the results obtained by the
authors in \cite{dangHomogenizationNondiluteSuspension2021}, where the
rigorous analysis of the homogenization  of  a  particulate flow
consisting of a non-dilute suspension of a viscous Newtonian fluid
with magnetizable particles was developed.  Here, the fluid is assumed
to be described by the Stokes flow and the particles are either
paramagnetic or diamagnetic. The coefficients of the corresponding
partial differential equations are locally periodic and a one-way
coupling between the fluid domain and the particles is also
assumed. Such one-way coupling has been observed in nature, see
\cite[Chapter 1]{davidsonIntroductionMagnetohydrodynamics2001}. For
details and information about the applications and literature on the
magnetic suspension, we turn to \cite{dangHomogenizationNondiluteSuspension2021}
and references cited therein; however, the mathematical formulation of
the considered problem is given in Section \ref{ss:setup}
below. References on the effective viscosity of a
suspension without the coupling with magnetic field include 
\cite{niethammerLocalVersionEinstein2020,gorbHomogenizationRigidSuspensions2014,hoferSedimentationInertialessParticles2018,mecherbetSedimentationParticlesStokes2019,gerard-varetAnalysisViscosityDilute2020,duerinckxCorrectorEquationsFluid2021,hainesProofEinsteinEffective2012,duerinckxSedimentationRandomSuspensions2021,duerinckxEinsteinEffectiveViscosity2020,duerinckxQuantitativeHomogenizationTheory2021,berlyandFictitiousFluidApproach2009a,berlyandHomogenizedNonNewtonianViscoelastic2004,desvillettesMeanfieldLimitSolid2008}.

In \cite{dangHomogenizationNondiluteSuspension2021}, a restrictive
assumption about the magnetic permeability of the suspension, denoted
by $\bfa$, was made. Here, the function $\bfa(\cdot)$ is locally
periodic and elliptic, where the latter means that $\lambda \II \leq
\bfa (x)$ and $ \norm{\bfa}_{L^\infty}\leq \Lambda$, for all $x \in \Omega$, with the
suspension domain $\Omega\subset \mathbb{R}^d$, $d=2,3$, including
both the ambient fluid and the particles, and $\lambda, \Lambda> 0$
given in \ref{cond:a-bound}-\ref{cond:a-elliptic} below.
The assumption on the function $\bfa$ made in \cite{dangHomogenizationNondiluteSuspension2021} is as follows: for a given $s \in (4,6]$, there exists a small number $\delta = \delta (\Lambda, d, \Omega) > 0$, for which the magnetic permeability $\bfa$ satisfies the following condition: 
\begin{equation}
\label{supinf}
\mathrm{ess} \sup \bfa - \mathrm{ess} \inf \bfa \le \delta.
\end{equation}
As a consequence of \eqref{supinf}, in \cite{dangHomogenizationNondiluteSuspension2021} we obtained the following gradient estimate for the magnetic potential $\varphi^{\varepsilon}$:  
\begin{equation}
\label{eq:493}
\int_{\Omega} \abs{\nabla \varphi^{\varepsilon}}^{s} \di \xx \le C
  \int_{\Omega} \abs{\kk}^{s} \di \xx,
\end{equation}
where the constant $C>0$ is independent of $\varepsilon$, $\varphi^{\varepsilon}$ and
{$\kk$}; with $0<\varepsilon \ll 1$ the scale of the microstructure, $\kk \in H^1({\Omega},\RR^d)$ divergence-free, satisfying the compatibility condition $\displaystyle\int_{\partial\Omega} \kk \cdot \nn_{\partial \Omega} \di \us=0$, and appearing in the Neumann boundary condition on $\partial\Omega$, the boundary of the domain $\Omega$, given by: 
\begin{equation} 
\label{NC}\left( \bfa \nabla \varphi^{\varepsilon} \right)\cdot \nn_{\partial \Omega}
  = \kk \cdot \nn_{\partial \Omega},
  \end{equation}where $\nn_{\partial\Omega}$ is the outward-pointing unit normal vector to $\partial \Omega$. The regularity result \eqref{eq:493} was then used in the derivation of the effective (or homogenized) response of the given suspension that was rigorously justified in \cite{dangHomogenizationNondiluteSuspension2021}.

The main goal of this paper is to relax the assumption
  \eqref{supinf} on the magnetic permeability $\bfa$. To achieve this, we
  consider the Dirichlet boundary condition given in \eqref{DC} below,
  rather than one given in \eqref{NC}, and obtain the Lipschitz estimate
  \eqref{eq:34}, instead of \eqref{eq:493}, for the gradient of the
  magnetic potential $\varphi^{\varepsilon}$, see
  \cref{sec:gradient-estimate-1} below. In this paper, we are able to remove the condition \eqref{supinf} and have $\bfa$ only required to be \emph{piecewise H\"{o}lder continuous}.
  Such relaxation will be based on the following observations:
 
\begin{itemize}
\item The De Giorgi-Nash-Moser
  estimate \cite[Theorem
  8.24]{gilbargEllipticPartialDifferential2001} states that the solutions of scalar equations are H\"{o}lder continuous.

\item If
  $1 > \varepsilon \ge \varepsilon_0$, for some $\varepsilon_0 > 0$,
  the uniform gradient bound \eqref{eq:34} can be obtained by
  the result of Li and Vogelius
  \cite{liGradientEstimatesSolutions2000}. The case when
  $\varepsilon_0 > \varepsilon > 0$ is resolved by the compactness method, which is discussed
  below.

\item Provided $\bfa$ is also 
  symmetric (this assumption is only necessary for the corrector
  results in \cref{cor:main}), the gradients of the solutions of the cell
  problems are in  $L^{\infty}(Y)$.
\end{itemize}

The main tools used in the proof of this theorem are (i) the regularity results of Li and Vogelius \cite{liGradientEstimatesSolutions2000}, and (ii) the celebrated \emph{compactness method}, which was first used in homogenization in the seminal works of Avellaneda and Lin
\cite{avellanedaCompactnessMethodsTheory1987,avellanedaCompactnessMethodsTheory1989}. Its
machinery and applications in homogenization are carefully explained
in \cite{prangeUniformEstimatesHomogenization2014}. In the context of homogenization, this method utilizes compactness in order to gain an improved regularity from a limiting equation via a proof by contradiction. This improvement of regularity is iterated and then used in a blow-up argument. Usually, the implementation of this method follows three steps, coined by Avellaneda and Lin
\cite{avellanedaCompactnessMethodsTheory1987,avellanedaCompactnessMethodsTheory1989} titled as {\it (i)} ``improvement'',  {\it (ii)} ``iteration'', and {\it (iii)} ``blow-up''.

The main contribution of this improved regularity result is that it will allow us to significantly widen the range of applicability of the results obtained in
\cite{dangHomogenizationNondiluteSuspension2021}. 

The outline of the paper is as follows. In
Section~\ref{sec:formulation}, the main notations are introduced and
the formulation of the fine-scale problem is discussed.
\cref{sec:gradient-estimate-1}, which provides an improved gradient
estimate for the magnetic potential, is stated and discussed in
Section~\ref{sec:gradient-estimate}.  
In Section~\ref{sec:InteriorEst}, we obtain the interior Lipschitz and H\"{o}lder estimates, which provide the foundation for the boundary and corrector estimates discussed in  Section~\ref{sec:bound-hold-estim}. With all the results at hand, we then present the proof of our main theorem, also in
Section~\ref{sec:bound-hold-estim}.
In
Section~\ref{sec:appl-magn-susp}, the homogenization results are
obtained and summarized in \cref{cor:main}.  The conclusions
are given in Section~\ref{sec:conclusions}. The classical Schauder estimate is recalled in \cref{sec:an-appendix}. 

\section{Formulation}
\label{sec:formulation}
\subsection{Notation}
For a measurable set $A$ and a measurable function $f\colon A \to \RR$, we define by $\abs{A}$ the measure of $A$ and $\displaystyle\dashint_A f(x)\di x \coloneqq \frac{1}{|A|}\int_A f(x)\di x$. 

Throughout this paper, the scalar-valued functions, such as the
pressure $p$, are written in usual typefaces, while vector-valued or
tensor-valued functions, such as the velocity $\uu$ and the Cauchy
stress tensor $\vec{\sigma}$, are written in bold.
Sequences are indexed by
numeric superscripts ($\phi^i$), while elements of vectors or tensors
are indexed by numeric subscripts ($x_i$). Finally, the Einstein
summation convention is used whenever applicable.
\subsection{Set up of the problem}
\label{ss:setup}
Consider $\Omega \subset \RR^d$, for $d \ge 2$, a simply connected and
bounded domain of class $C^{1,\alpha},~ 0< \alpha <1$, and let
$Y\coloneqq (0,1)^d$ be the unit cell in $\RR^d$. The unit cell $Y$ is
decomposed into:
$$Y=Y_s\cup Y_f \cup \Gamma,$$
where $Y_s$, representing the magnetic
inclusion, and $Y_f$, representing the fluid domain, are open sets in
$\mathbb{R}^d$, and $\Gamma$ is the closed $C^{1,\alpha}$ interface that
separates them.
Let $i = (i_1, \ldots, i_d) \in \ZZ^d$ be a vector of indices and $\{\ee^1, \ldots, \ee^d\}$ be the canonical basis of $\RR^d$. 
For a fixed small $\varepsilon > 0,$ we define the dilated sets: 
\begin{align*}
    Y^\varepsilon_i 
    \coloneqq \varepsilon (Y + i),~~
    Y^\varepsilon_{i,s}
    \coloneqq \varepsilon (Y_s + i),~~
    Y^\varepsilon_{i,f}
    \coloneqq \varepsilon (Y_f + i),~~
    \Gamma^\varepsilon_i 
    \coloneqq \partial Y^\varepsilon_{i,s}.
\end{align*}
Typically, in homogenization theory, the positive number $\varepsilon
\ll 1$ is referred to as the {\it size of the microstructure}. The
effective (or {\it homogenized}) response of the given suspension
corresponds to the case $\varepsilon=0$.

We denote by $\nn_i,~\nn_{\Gamma}$ and $\nn_{\partial \Omega}$ the unit normal vectors to $\Gamma^{\varepsilon}_{i}$ pointing outward $Y^\varepsilon_{i,s}$, on $\Gamma$ pointing outward $Y_{s}$ and on $\partial \Omega$ pointing outward, respectively; and also, we denote by $\di \hmeas$ the $(d-1)$-dimensional Hausdorff measure.
In addition, we define the sets:
\begin{align*}
    I^{\varepsilon} 
    \coloneqq \{ 
    i \in \ZZ^d \colon Y^\varepsilon_i \subset \Omega
    \},~~
    \Omega_s^{\varepsilon} 
    \coloneqq \bigcup_{i\in I^\varepsilon}
Y_{i,s}^{\varepsilon},~~
    \Omega_f^{\varepsilon} 
    \coloneqq \Omega \setminus \Omega_s^{\varepsilon},~~
    \Gamma^\varepsilon 
    \coloneqq \bigcup_{i \in I^\varepsilon} \Gamma^\varepsilon_i,
\end{align*}
see \cref{fig:1}.

\begin{figure}[b]
  \sidecaption
\includegraphics[scale=.25]{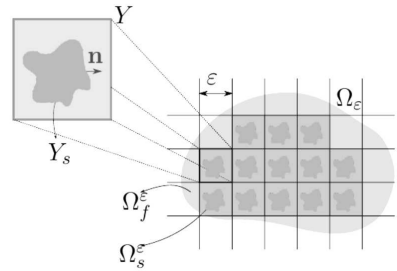}
\caption{Reference cell $Y$ and domain $\Omega$.}
\label{fig:1}
\end{figure}

The magnetic permeability $\bfa$ is a $d \times d$
matrix satisfying the following conditions:
\begin{enumerate}[label=(A{\arabic*}),ref=\textnormal{(A{\arabic*})}]
\item \label{cond:a-periodic} $Y-$periodicity: for all $z
  \in \RR^d$, for all $m \in \ZZ$, and for all $k \in \{ 1,\ldots, d \}$ we have: 
\begin{equation*}
\bfa (z + m \ee^k) = \bfa (z).
\end{equation*}
  \item\label{cond:a-bound} Boundedness and measurability: there exists $\Lambda > 0$ such that: 
\begin{equation*}
\norm{\bfa}_{L^{\infty} \left( \RR^d \right)} \le \Lambda.
\end{equation*}
  \item\label{cond:a-elliptic} Ellipticity: there exists
    $\lambda > 0$ such that for all $\xi \in \RR^d$, for all $x \in
    \RR^d$, we have: 
\begin{equation*}
\bfa (x) \xi \cdot \xi \ge \lambda \abs{\xi}^2.
\end{equation*}
\end{enumerate}
Denote by $\fraM (\lambda,\Lambda)$ the set of matrices that satisfy
\ref{cond:a-bound}-\ref{cond:a-elliptic}, and
$\fraM_{\per}(\lambda,\Lambda)$ the subset of matrices in $\fraM(\lambda,\Lambda)$
that also satisfy \ref{cond:a-periodic}. 

\section{Statement and discussion of the main result}
\label{sec:main-result}
The main result of this paper is summarized in the following theorem.
\label{sec:gradient-estimate}
\begin{theorem}[Global Lipschitz Estimate]
\label{sec:gradient-estimate-1}
Let $\Omega$ be a bounded $C^{1,\alpha}$ domain, $g \in C^{1,\alpha'}(\partial\Omega)$
and $f \in L^{\infty}( {\Omega})$, where $0 < \alpha'< \alpha < 1$. Suppose that
$\bfa \in \fraM_{\per}(\lambda,\Lambda)$ is piecewise
$C^{\alpha}$-continuous. There exists 
$C=C(\alpha,\alpha',\lambda, \Lambda, d,\Omega) > 0$ such that, for all $\varepsilon > 0$, the (unique) solution $\varphi^{\varepsilon}$ of:
\begin{subequations}
  \label{eq:4}
  \begin{align}
    \label{eq:5}
    -\Div \left[ \bfa \left( \frac{x}{\varepsilon} \right) \nabla
    \varphi^{\varepsilon} \right]
    &= f, \quad &&\text{ in }\Omega\\
    \varphi^{\varepsilon}
    &= g, \quad &&\text{ on }\partial \Omega,\label{DC}
  \end{align}
\end{subequations}
satisfies:
\begin{align}
\label{eq:34}
  \norm{\nabla \varphi^{\varepsilon}}_{L^{\infty}(\Omega)}
  &\le C \left( \norm{g}_{C^{1,\alpha'}(\partial\Omega)} + \norm{f}_{L^{\infty}( {\Omega})} \right).
\end{align}
\end{theorem}

\begin{remark}\label{sec:impr-grad-estim}
  For each $\varepsilon > 0$, let
  $N_{\varepsilon}$ be the number of subdomains inside $\Omega$ such
  that in each of them the function $\bfa$ is
  $C^{\alpha}$-continuous. Denote those subdomains by $D_m,~1\le m \le N_{\varepsilon}$. Then, for
  $0 < \alpha' < \min \left\{ \alpha, \frac{\alpha}{(\alpha+1)d}
  \right\}$,
by \cite[Corollary 1.3]{liGradientEstimatesSolutions2000}, one has:
\begin{align}
\label{eq:6}
  \norm{\nabla \varphi^{\varepsilon}}_{L^{\infty}(\Omega)}
  \le C \left( \norm{g}_{C^{1,\alpha'}(\partial \Omega)} + \norm{f}_{L^{\infty}(\Omega)}\right), 
\end{align}
where $C$ depends on $\Omega, d, \alpha, \alpha', \lambda, \Lambda, \norm{\bfa}_{C^{\alpha'}(\overline{D_m},\RR^{d\times d})}$ and the
$C^{1,\alpha}$-modulus of $\cup_{m=1}^{N_{\varepsilon}} \partial
D_m$ (defined in page 92
  \cite{liGradientEstimatesSolutions2000}). As $\varepsilon \to 0,$ the number $N_\varepsilon$ increases, while the sizes of the subdomains decrease, which leads to the blow-up of the  $C^{1,\alpha}$-modulus. Therefore, estimate \eqref{eq:6} is not
uniform in $\varepsilon$.

However, if $\varepsilon_0 \le \varepsilon \le 1$ for some constant $\varepsilon_0 > 0$,
then one can control the number, size, distance and
  $C^{1,\alpha}$-modulus  of the subdomains in $\Omega$, uniformly with respect
  to $\varepsilon$. Note
  that the upper and lower bounds of those quantities are positive and independent of $\varepsilon$. Thus, by \cite[Corollary 1.3]{liGradientEstimatesSolutions2000}, there exists $C$ independent of $\varepsilon$ such that \eqref{eq:34} holds when $\varepsilon_0 \le \varepsilon \le 1$. Therefore, \cref{sec:gradient-estimate-1} will be proven, if one can specify a constant $\varepsilon_0 > 0$ such that \eqref{eq:34} holds for $0 < \varepsilon < \varepsilon_0.$
\end{remark}

Our proof of \cref{sec:gradient-estimate-1} follows the classical steps in regularity theory: (i) derive an interior Lipschitz estimate, (ii) derive a boundary Lipschitz estimate, and finally (iii) combine the estimates in (i) and (ii) to obtain the global Lipschitz estimate.
Step (i) is obtained via the \textit{compactness method}. For step (ii), we additionally need to establish the following preliminary results:
\begin{itemize}
\item Interior and boundary H\"{o}lder estimates, see \cref{sec:lipschitz-estimate,sec:lipschitz-estimate-5,sec:lipschitz-estimate-3}.
\item Estimates for the Green's function, which are obtained using the H\"{o}lder bounds above, see \cref{sec:bound-hold-estim-2}.  The existence of the Green's function for scalar uniformly elliptic equations is established in
  \cite{littmanRegularPointsElliptic1963,gruterGreenFunctionUniformly1982,kenigEllipticEquationLuk1985}.
\item Estimates for the Dirichlet boundary corrector, see \cref{sec:lipschitz-estimate-2}.
\end{itemize}

If the coefficient $\bfa$ is not H\"{o}lder continuous, the classical results in \cite{avellanedaCompactnessMethodsTheory1987,prangeUniformEstimatesHomogenization2014,shenPeriodicHomogenizationElliptic2018} can not be applied directly. Nevertheless,  some of their proofs can be adapted for the case at hand. In those situations, we will explicitly point out what needs to be  modified in their proofs, in order to relax the continuity assumption on the coefficient matrix $\bfa$.

\section{Interior estimates}
\label{sec:InteriorEst}
We start with an estimate for homogenized equations, i.e. the equations with constant coefficients, which are limits of some fine-scale problems.

\begin{lemma}
\label{interation-lem}
Let
$\lambda > 0,$ $\Lambda > 0$, $\gamma > 0$ and $0 < \mu < \frac{1}{2}$ be fixed. For each constant
matrix $\bfb \in \fraM (\lambda,\Lambda)$,
$h \in L^{d + \gamma}( {B}(x_0,1))$, with
$\norm{h}_{L^{d + \gamma}( {B}(x_0,1))} \le 1$, there exists $\theta = \theta (\gamma,\mu,\lambda,\Lambda,d) > 0$ such that if
$\phi \in H^1(B(x_0,1))$ satisfies:
\begin{align*}
-\Div \left( \bfb \nabla \phi \right) =  h \text{ in } B(x_0,1),
\end{align*}
then the following estimate holds:
\begin{align}
\label{eq:45}
\sup_{\abs{x-x_0} < \theta} \abs{\phi(x) - \phi(x_0) - (x-x_0)\cdot
  \dashint_{B(x_0,\theta)}\nabla \phi (z) \di z }< \theta^{1+3\mu/4}.
\end{align}
\end{lemma}

\begin{proof}
By the classical Schauder estimate for the scalar equation with constant
coefficients (\cref{sec:an-appendix-1}),
we have $\phi \in C^{1,\mu} \left( B \left( x_0, 1/4 \right)
\right)$ and:
\begin{align}
\label{eq:44}
  \begin{split}
    \norm{\phi}_{C^{1,\mu}\left( B \left( x_0,1/4 \right) \right)}
    &\le C (\gamma,\mu,\lambda,\Lambda,d) \left( \norm{h}_{L^{d+\gamma}\left( B(x_0,3/8)
        \right)} + \norm{\phi}_{H^1(B(x_0,3/8))} \right)\\
    &\le C (\gamma,\mu,\lambda,\Lambda,d)
    \norm{h}_{L^{d + \gamma}( {B}(x_0,1))}\\
    &\le C(\gamma,\mu,\lambda,\Lambda,d).
  \end{split}
\end{align}
For $0 < \theta < \frac{1}{4}$ and $\abs{x - x_0} < \theta$, there
exist $z_x$ such that:  
\begin{align*}
&\abs{\phi(x) - \phi(x_0) - (x-x_0)\cdot \dashint_{B(x_0,\theta)} \nabla
  \phi (z) \di z}\\
  &\quad= \abs{\frac{x-x_0}{\abs{B(x_0,\theta)}} \cdot \int_{B(x_0,\theta)}
    \left( \nabla \phi(z_x) - \nabla \phi (y) \right)\di y }\\
  &\quad\le C (\gamma,\mu,\lambda,\Lambda,d) \abs{x-x_0}^{1+\mu}\\
  &\quad\le C (\gamma,\mu, \lambda,\Lambda,d) \theta^{1+\mu}.
\end{align*}
Choosing $\theta$ small enough so that $C (\gamma,\mu, \lambda,\Lambda,d)
\theta^{1+\mu} < \theta^{1+ 3\mu/4}$, we obtain \eqref{eq:45}.
\end{proof}

The fact that $\theta$ doesn't depend on the matrix $\bfb$ or the source term $h$ is crucial for the contradiction argument in \cref {sec:lipschitz-estimate} below.
We now state the interior Lipschitz estimate. Note
that, here, $\bfa$ is not necessarily H\"{o}lder continuous in the domain
$\Omega$.

\begin{proposition}[Interior Lipschitz Estimate I]
\label{sec:lipschitz-estimate}
Suppose that $\bfa \in \fraM_{\per}(\lambda,\Lambda)$ and
$f \in L^{\infty}(\Omega)$. 
Fix
$x_0\in\Omega$ and $R > 0$, so that $B(x_0,R)\subset \Omega$. 
There exist
$\varepsilon_0 = \varepsilon_0 (\lambda,\Lambda,R, d) > 0$ and $C = C (\lambda,\Lambda,R, d) > 0$ such that, for all $0 < \varepsilon < \varepsilon_0$ and for every
weak solution $\varphi^{\varepsilon} \in H^1(B(x_0,R))$ of the equation
$-\Div \left[ \bfa \left( \frac{x}{\varepsilon} \right)\nabla
  \varphi^{\varepsilon} \right] = f$ in $B(x_0,R)$,  the following estimate holds:
\begin{align}
\label{eq:35}
  \norm{\nabla \varphi^{\varepsilon}}_{L^{\infty}(B(x_0,R/2))}
  \le C  \left( \norm{\varphi^{\varepsilon}}_{L^\infty(B(x_0,R))} + \norm{f}_{L^{\infty}( {B}(x_0,R))} \right).
\end{align}

\end{proposition}

\begin{proof}
  By dilation, we may assume that $R = 1$.  Fix $0< \mu < \frac{1}{2}$. We prove, by the compactness method, that
there exists  $\varepsilon_0=\varepsilon_0(\lambda,\Lambda,d,\mu)$ so that \eqref{eq:35} holds for all
$0 < \varepsilon < \varepsilon_0$. To do this, we only need to show
that there exists $C > 0$, independent of $\varepsilon$, such that:
\begin{align}
  \label{eq:60}
\max \left\{ \norm{\varphi^{\varepsilon}}_{L^{\infty}(B(x_0,1))},
  \norm{f}_{L^{\infty}( {B}(x_0,1))} \right\} \le 1 \text{ implies } \norm{\nabla
  \varphi^{\varepsilon}}_{L^{\infty}(B(x_0,1/2))} \le C.
\end{align}
Let $\vec{\omega} \coloneqq (\omega^1,\ldots, \omega^d)$, where $\omega^i
\in H^1_{\per}(Y)/\RR,~1\le i \le d,$ is the solution of the \textit{cell problem}:
\begin{equation}
\label{eq:pp453}
-\Div_{\yy} \left[ 
{\bfa(\yy) }\left(\ee^i +\nabla_{\yy} \omega^i(\yy)
  \right)  \right]
  = 0 \text{ in } Y.
  \end{equation}
  \begin{enumerate}[wide]
  \item \emph{Improvement.} \quad In this step, we prove by contradiction
    that:

For fixed $0< \mu < \frac{1}{2}$, there exist $\theta$ and $\varepsilon^*$, with $0 < \theta <\frac{1}{4}, ~0<\varepsilon^* < 1$, depending on
$\lambda,\Lambda,d$ and $\mu$, such that if
$\bfa \in \fraM_{\per}(\lambda,L)$,
$f \in L^{\infty}( {B}(x_0,1))$,
$\varphi^{\varepsilon} \in H^1(B(x_0,1))$ satisfy:
\begin{subequations}
\begin{align}
\label{eq:42}
-\Div \left[ \bfa \left( \frac{x}{\varepsilon} \right) \nabla
  \varphi^{\varepsilon} \right]
  &= f \text{ in } B(x_0,1),\\
  \label{eq:50}
  \max \left\{ \norm{\varphi^{\varepsilon}}_{L^{\infty}(B(x_0,1))},
  \norm{f}_{L^{\infty}( {B}(x_0,1))} \right\}
  &\le 1,
\end{align}
\end{subequations}
then, for all $0 < \varepsilon < \varepsilon^*$, we have: 
\begin{align}
\label{eq:43}
  \sup_{\abs{x-x_0} < \theta}\abs{\varphi^{\varepsilon}(x)-
  \varphi^{\varepsilon}(x_0)- \left[ x -x_0 + \varepsilon \vec{\omega} \left(
  \frac{x}{\varepsilon} \right) \right] \cdot
  \dashint_{B(x_0,\theta)}\nabla
  \varphi^{\varepsilon}(z) \di z}\le \theta^{1+\mu/2},
\end{align}
where $\vec{\omega}$ solves \eqref{eq:pp453}. 

Take $\theta$ as in \eqref{eq:45} of \cref{interation-lem}. By contradiction, suppose there exist
sequences: 
\begin{align*}
\varepsilon_n \to 0, \quad \bfa_n \in \fraM_{\per} (\lambda,L), \quad
  f_n \in L^{\infty}\left( B(x_0,1) \right), \text{ and } \varphi_n
  \in H^1 (B(x_0,1))
\end{align*}
satisfying: 
\begin{align}
\label{eq:46}
-\Div \left[ \bfa_n \left( \frac{x}{\varepsilon_n}  \right) \nabla \varphi_n
  \right]
  &= f_n \text{ in } B(x_0,1),\\
  \label{eq:47}
  \max \left\{
  \norm{\varphi_n}_{L^{\infty}(B(x_0,1))},
  \norm{f_n}_{L^{\infty}( {B}(x_0,1))}\right\}
  &\le 1,
\end{align}
such that: 
\begin{align}
\label{eq:48}
  \sup_{\abs{x - x_0} < \theta} \abs{ \varphi_n (x) - \varphi_n(x_0) -
  \left[  x - x_0 +\varepsilon_n \vec{\omega} \left( \frac{x}{\varepsilon_n}
  \right) \right] \cdot \dashint_{B(x_0,\theta)}\nabla \varphi_n (z) \di z}
  > \theta^{1+\mu/2}.
\end{align}
Let $\calA_n \in \fraM (\lambda,\Lambda)$ denote the effective matrix corresponding to
$\bfa_n$. By the Banach-Alaoglu theorem, the Caccioppoli inequality, the Rellich-Kondrachov theorem  and the Schauder theorem (see, e.g. \cite[Theorem 4.4]{giaquintaIntroductionRegularityTheory2012} and \cite[Theorem 3.16, 6.4 and 9.16]{brezisFunctionalAnalysisSobolev2011}), there exist functions $\varphi_0 \in L^2(B(x_0,1)),$ $f_0 \in L^{\infty}(B(x_0,1))$ and a constant matrix $\bfa_0 \in \fraM (\lambda,\Lambda)$ such that,  up to subsequences, we have: 
\begin{align*}
  \varphi_n
  &\rightharpoonup \varphi_0 \text{ in } L^2(B(x_0,1))\\
  f_n
  &\overset{*}{\rightharpoonup} f_0 \text{ in } L^{\infty}(B(x_0,1))\\
  \varphi_n
  &\rightharpoonup \varphi_0 \text{ in } H^1(B(x_0,1/2))\\
  f_n
  &\to f_0 \text{ in } H^{-1}(B(x_0,1))\\
  \calA_n
  &\to \bfa_0 .
\end{align*}
By \cite[Theorem 13.4 (iii)]{cioranescuIntroductionHomogenization1999}
or \cite[Theorem 2.3.2]{shenPeriodicHomogenizationElliptic2018}, we
have $\varphi_0$ is the solution of: 
\begin{align}
\label{eq:49}
-\Div \left( \bfa_0 \nabla \varphi_0 \right) = f_0 \text{ in } B(x_0,1/2).
\end{align} 
Fix $x \in B(x_0,1)$ and let $U \subset B(x_0,1)$ be an open
neighborhood of $x$. By the De Giorgi-Nash-Moser Theorem \cite[Theorem
8.24]{gilbargEllipticPartialDifferential2001}, there exists $0 < \beta = \beta (d, \lambda/\Lambda) <
1$ such that: 
\begin{align*}
  \norm{\varphi_n}_{C^{\beta}(\overline{U})}
  \le C \left(\norm{\varphi_n}_{L^{\infty}(B(x_0,1))} +
  \norm{f_n}_{L^{\infty}( {B}(x_0,1))}\right) \le 2C.
\end{align*}
By the Arzela-Ascoli Theorem, up to a subsequence, $\varphi_n$ uniformly
converges to $\varphi^{*}$ in $C(U)$ for some $\varphi^{*}$. Since $\varphi_n \rightharpoonup \varphi_0$ in $L^2(B(x_0,1))$, we
conclude that $\varphi^{*} = \varphi_0$ a.e. in $U$.  Therefore, $\lim_{n \to \infty} \varphi_n(x) = \varphi_0(x)$ a.e. in
$B(x_0,1)$. Letting $n \to \infty$ in \eqref{eq:47}, the argument above and \cite[Theorem 3.13]{brezisFunctionalAnalysisSobolev2011} yield: 
\begin{align*}
 \max \left\{
  \norm{\varphi_0}_{L^{\infty}(B(x_0,1))},
  \norm{f_0}_{L^{\infty}( {B}(x_0,1))}\right\}
  &\le 1,
\end{align*}
which, together with \eqref{eq:49}, implies that \eqref{eq:45} still holds for
$\phi = \varphi_0$ (observe that, from \eqref{eq:44}, shrinking the domain from $B(x_0,1)$ to $B(x_0,1/2)$ does not affect estimates), that is:
\begin{align}
\label{eq:51}
\sup_{\abs{x-x_0} < \theta} \abs{\varphi_0(x) - \varphi_0(x_0) - (x-x_0)\cdot
  \dashint_{B(x_0,\theta)}\nabla \varphi_0 (z) \di z} < \theta^{1+3\mu/4}.
\end{align}
On the other hand, letting $n\to \infty $ in \eqref{eq:48} and since $\norm{\vec{\omega}}_{L^{\infty}(Y)}<\infty$, we obtain: 
\begin{align*}
\sup_{\abs{x-x_0} < \theta} \abs{\varphi_0(x) - \varphi_0(x_0) - (x-x_0)
  \cdot\dashint_{B(x_0,\theta)}\nabla \varphi_0 (z) \di z} \ge \theta^{1+\mu/2},
\end{align*}
which contradicts \eqref{eq:51}, since $0 < \theta < \frac{1}{4}.$

  \item \emph{Iteration} \quad Let $0 < \varepsilon < \varepsilon^*$. Direct evaluation yields that:
\begin{align*}
  P^{\varepsilon} (x)
  \coloneqq &\frac{1}{\theta^{1+\mu/2}} \left\{
  \varphi^{\varepsilon}(\theta x) - \varphi^{\varepsilon}(\theta x_0)
  \vphantom{\dashint_{B(x_0,\theta)}}\right.\\
  &\qquad\left.
  - \left[ \theta (x - x_0) + \varepsilon \vec{\omega} \left( \frac{\theta
  x}{\varepsilon} \right)\right] \cdot \dashint_{B(x_0,\theta)}
  \nabla \varphi^{\varepsilon} (z) \di z
  \right\}
\end{align*}
solves the following equation:
\begin{align}
  \label{eq:53}
-\Div \left[ \bfa \left( \frac{\theta x}{\varepsilon} \right) \nabla
  P^{\varepsilon}(x) \right] = \tilde{f} \text{ in } B(x_0,1),
\end{align}
where $\tilde{f} \coloneqq \theta^{1-\mu/2}f(\theta x)$.  
Moreover, by \eqref{eq:43} and \eqref{eq:50}, we have:
$$\norm{P^{\varepsilon}}_{L^{\infty}(B(x_0,1))} \le 1 \,\text{ and}\, 
\norm{\tilde{f}}_{L^{\infty}( {B}(x_0,1))} \le 1,$$ so by using
\eqref{eq:43} again, we obtain: 
\begin{align}
\label{eq:52}
  \sup_{\abs{x-x_0} < \theta}\abs{P^{\varepsilon}(x)-
  P^{\varepsilon}(x_0)- \left[  x -x_0 + \frac{\varepsilon}{\theta}
  \vec{\omega} \left(  \frac{\theta x}{\varepsilon} \right) \right] \cdot
  \dashint_{B(x_0,\theta)}\nabla
  P^{\varepsilon}} (z) \di z \le \theta^{1+\mu/2}.
\end{align}
From \eqref{eq:52} and scaling down, we have: \begin{align*}
  \sup_{\abs{x-x_0} < \theta^2}\abs{\varphi^{\varepsilon}(x)-
  \varphi^{\varepsilon}(x_0)- \left( x -x_0\right)\cdot
  a_2^{\varepsilon} + \varepsilon b_2^{\varepsilon}}\le \theta^{2(1+\mu/2)},
\end{align*}
where:
\begin{align}
\label{eq:59}
  \begin{split}
    a_2^{\varepsilon}
    &\coloneqq\dashint_{B(x_0,\theta)}\nabla
    \varphi^{\varepsilon} (z) \di z + \theta^{\mu/2}~
    \dashint_{B(x_0,\theta)}\nabla P^{\varepsilon}(z) \di z,\\
    b_2^{\varepsilon}(y)
    &\coloneqq \vec{\omega} \left( y
    \right) \cdot\left( \dashint_{B(x_0,\theta)}\nabla
    \varphi^{\varepsilon}(z) \di z + \theta^{\mu/2}
    \dashint_{B(x_0,\theta)}\nabla P^{\varepsilon}(z) \di z \right), \text{ for
  } y \coloneqq \frac{x}{\varepsilon} \in Y.
  \end{split}
\end{align}
By the De Giorgi-Nash-Moser estimate and Caccioppoli inequality \cite[Lemma
C.2]{armstrongQuantitativeStochasticHomogenization2019}, there exists
a constant $C$, depending only on $\lambda,\Lambda$ and $d$, such
that:
\begin{align*}
  \norm{\vec{\omega}}_{L^{\infty}(Y)}
  \le C,\qquad
  \abs{\dashint_{B(x_0,\theta)}\nabla \varphi^{\varepsilon}(z) \di z}
  \le C/\theta,\qquad
    \abs{\dashint_{B(x_0,\theta)}\nabla P^{\varepsilon} (z) \di z }
  \le C/\theta.
\end{align*}
Therefore, \eqref{eq:59} implies that:
\begin{align*}
  \abs{a_2^{\varepsilon}}
  &\le (C/\theta) \left( 1+\theta^{\mu/2} \right),\\
  \norm{b_2^{\varepsilon}}_{L^{\infty}(Y)}
  &\le (C/\theta) \left( 1+\theta^{\mu/2} \right).
\end{align*}
Reiterating this process, we obtain that there exists $C=C(\gamma,\lambda,\Lambda,d,\mu) > 0$ such that:
\begin{align}
  \label{eq:55}
  \begin{split}
  \abs{a_k^{\varepsilon}}
  &\le (C/\theta) \left(1+ \theta^{\mu/2}+ \cdots +
    \theta^{(k-1)\mu/2}\right),\\
  \norm{b_k^{\varepsilon}}_{L^{\infty}(Y)}
  &\le (C/\theta) \left(1+ \theta^{\mu/2}+ \cdots +
    \theta^{(k-1)\mu/2}\right),
  \end{split}
\end{align}
and: 
\begin{align}
\label{eq:54}
\sup_{\abs{x-x_0} < \theta^k}\abs{\varphi^{\varepsilon}(x)-
  \varphi^{\varepsilon}(x_0)- \left( x -x_0\right)\cdot
  a_k^{\varepsilon} + \varepsilon b_k^{\varepsilon}}\le \theta^{k(1+\mu/2)}.
\end{align}

\item \emph{Blow-up} \quad  Let $\varepsilon_0 \coloneqq \min \left\{\varepsilon^*,\frac{1}{5\sqrt{d}} \right\}$ and $0 < \varepsilon < \varepsilon_0.$

  Choose $k$ such that $\theta^{k+1} \le 4\varepsilon   \sqrt{d} <
  \theta^k$. Then, from \eqref{eq:54}, there exists $C = C(\theta,d) > 0$ so that: 
\begin{align}
\label{eq:56}
\sup_{\abs{x-x_0} < 4\varepsilon   \sqrt{d}}\abs{\varphi^{\varepsilon}(x)-
  \varphi^{\varepsilon}(x_0)- \left( x -x_0\right)\cdot
  a_k^{\varepsilon} + \varepsilon b_k^{\varepsilon}}
  \le C\varepsilon^{1+\mu/2}, 
\end{align}
which, together with \eqref{eq:55}, leads to: 
\begin{align}
\label{eq:57}
\norm{\varphi^{\varepsilon} -
  \varphi^{\varepsilon}(x_0)}_{L^{\infty}(B(x_0,4\varepsilon   \sqrt{d}))} \le C
  \varepsilon. 
\end{align}
Denote by $z_0^{\varepsilon}$ the center of the cell $Y^\varepsilon_i$ containing $x_0$, and define:  \begin{align*}
    v^{\varepsilon} (x)
    \coloneqq \frac{1}{\varepsilon} \left[ 
    \varphi^{\varepsilon}(\varepsilon x + z_0^\varepsilon ) - \varphi^\varepsilon (\varepsilon x_0 + z_0^\varepsilon )
    \right], \quad x \in \Omega.
\end{align*}
Then, $\nabla v^{\varepsilon} (x) = \nabla \varphi^{\varepsilon} \left(\varepsilon x + z_0^{\varepsilon}\right)$ and, moreover, $v^{\varepsilon}$ solves: 
\begin{align}
\label{eq:blow-up}
    -\Div \left[ \bfa \left( x + \frac{z_0^{\varepsilon}}{\varepsilon} \right) \nabla v^{\varepsilon}(x)\right] 
    = \varepsilon f (\varepsilon x + z_{0}^{\varepsilon}) \quad \text{ in } B \left(0, 3\sqrt{d} \right).
\end{align}
Observe that: 
\begin{align}
  \label{eq:8}
  \begin{split}
    \frac{1}{\varepsilon} \left( 
    B \left(x_{0},\varepsilon \sqrt{d}\right) - z_{0}^{\varepsilon}
    \right) 
    &\subset B \left(0, 2\sqrt{d} \right) \\
    &\subset B \left(0, 3\sqrt{d} \right) 
    \subset \frac{1}{\varepsilon} \left( 
    B \left(x_{0},4\varepsilon \sqrt{d}\right) - z_{0}^{\varepsilon}
  \right).
  \end{split}
\end{align}
Applying \cite[Theorem 1.1]{liGradientEstimatesSolutions2000} to \eqref{eq:blow-up}, we obtain that there exists a constant $C > 0$, independent of $\varepsilon$ and $x_0$, such that:
\begin{align}
\label{eq:7}
\begin{split}
    &\norm{\nabla v^{\varepsilon}}_{L^{\infty}\left( \frac{1}{\varepsilon} \left( 
    B \left(x_{0},\varepsilon \sqrt{d}\right) - z_{0}^{\varepsilon}
    \right) \right) }\\
    &\quad\le 
    \norm{\nabla v^{\varepsilon}}_{L^{\infty} \left(
    B \left(0, 2\sqrt{d} \right)
    \right)}\\
    &\quad\le 
    C \left[
    \norm{v^{\varepsilon}}_{L^{\infty}\left(
    B \left(0, 3\sqrt{d} \right)\right)}
    +\sup_{x \in B \left(0, 3\sqrt{d} \right)} \abs{\varepsilon f (\varepsilon x + z_{0}^{\varepsilon}) }
    \right]\\
    &\quad\le 
    C \left[
    \norm{v^{\varepsilon}}_{L^{\infty}\left( \frac{1}{\varepsilon} \left( 
    B \left(x_{0},4\varepsilon \sqrt{d}\right) - z_{0}^{\varepsilon}
    \right) \right) }
    + \norm{f}_{L^{\infty}\left( \frac{1}{\varepsilon} \left( 
    B \left(x_{0},4\varepsilon \sqrt{d}\right) - z_{0}^{\varepsilon}
    \right) \right) }
    \right]
\end{split}
\end{align}
Scaling down \eqref{eq:7} and using \eqref{eq:57}, we obtain:
\begin{align*}
    &\norm{\nabla \varphi^{\varepsilon}}_{L^{\infty} \left(
    B \left(x_0, \varepsilon \sqrt{d} \right)
    \right)}\\
    &\le 
    C \left[ \frac{1}{\varepsilon}
    \norm{\varphi^{\varepsilon} - \varphi^{\varepsilon} (x_0)}_{L^{\infty}\left( \left( 
    B \left(x_{0},4\varepsilon \sqrt{d}\right)
    \right) \right) }
    + \norm{f}_{L^{\infty}\left( 
    B \left(x_{0},4\varepsilon \sqrt{d}\right) \right) }
    \right]
    \le C,
\end{align*}
where $C > 0$ is independent of $x_0$ and $\varepsilon$.

\end{enumerate}
\end{proof}

\begin{remark}
\label{sec:stat-disc-main}
 Under stronger smoothness assumptions on the coefficient $\bfa$,
similar estimates to \eqref{eq:6} are proved in the literature. In particular, if $\bfa$
is in VMO$(\RR^d)$, the real-variable method of L. Caffarelli and
I. Peral \cite{caffarelliW1PestimatesElliptic1998} yields an uniform
$W^{1,p}-$estimate; on the other hand, if $\bfa$ is H\"{o}lder
continuous, then one has the uniform Lipschitz estimate. Those results
hold also for elliptic systems, and even for Neumann boundary
condition. We refer the reader to
\cite{shenPeriodicHomogenizationElliptic2018,shenBoundaryEstimatesElliptic2017,shenW1EstimatesElliptic2008,kenigHomogenizationEllipticSystems2013,liEstimatesEllipticSystems2003,avellanedaCompactnessMethodsTheory1989,avellanedaCompactnessMethodsTheory1987,prangeUniformEstimatesHomogenization2014}
and references cited therein.

However, in this paper, we focus on the case when $\bfa$ is only piecewise H\"{o}lder
continuous. A similar argument as in the papers cited above, together with the regularity theorem of Li and Vogelius \cite[Theorem
1.1]{liGradientEstimatesSolutions2000} yield the interior Lipschitz estimate as showed in
\cref{sec:lipschitz-estimate}. Moreover, some additional care is needed
to ensure that the constant $C$ in \eqref{eq:7} is independent of \emph{both} $\varepsilon$ and $x_0$. 
In the Blow-up step of the proof above, one may try to let
\begin{align}
\label{eq:10}
s^{\varepsilon}(x) \coloneqq \frac{1}{\varepsilon}
\varphi^{\varepsilon}(\varepsilon x + x_0)
\end{align}
so that
\begin{align}
\label{eq:9}
-\Div
\left[ \bfa \left( x \right) \nabla
    s^{\varepsilon}(x) \right] = \varepsilon f(\varepsilon x + x_0),
\end{align}
then by applying \cite[Theorem 1.1]{liGradientEstimatesSolutions2000}, one obtains
$$
\norm{\nabla s^{\varepsilon}}_{L^{\infty} \left( B \left( 0,
      \frac{1}{2\varepsilon_0} \right) \right)} \le C' \left[
  \norm{s^{\varepsilon}}_{L^{\infty} \left( B \left( 0,
        \frac{1}{\varepsilon_0} \right) \right)} +
  \norm{f}_{L^{\infty} \left( B \left( 0, \frac{1}{\varepsilon_0}
      \right) \right)} \right].
$$
However, \emph{$C'$ indeed depends on $x_0$}. The reason is that, when one
shifts $x_0$ in the scaling \eqref{eq:10}, one
also changes the $C^{1,\alpha}$-modulus of the subdomains, where
the latter, in the context of our problem, are generated by taking
intersections of the ball centered at $x_0$ and the heterogenous
domain.  In short, we do not have the uniform control of the
subdomains when using the scaling \eqref{eq:10} for arbitrary $x_0$, see
\cref{fig:2}. In order to circumvent the dependence on $x_0$, we use a different scaling and combine with a geometric argument, as demonstrated in the proof of \cref{sec:lipschitz-estimate}.

\begin{figure}
  \sidecaption
\includegraphics[scale=.450]{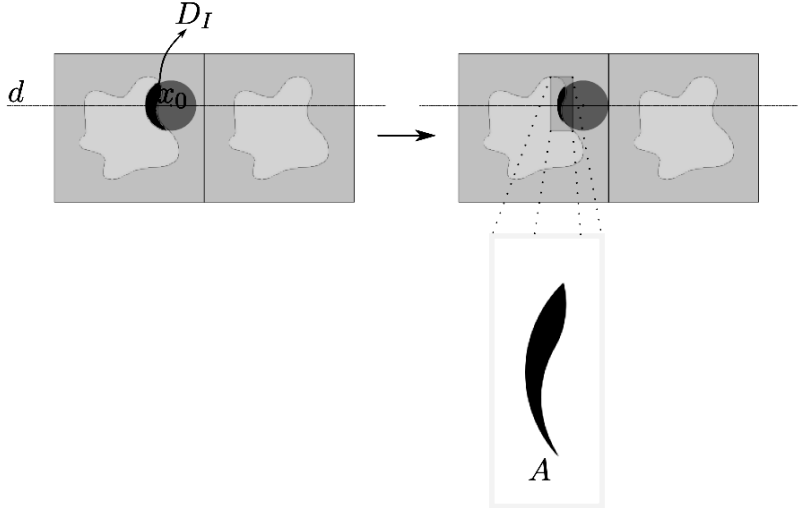}
\caption{As the center $x_0$ of the ball $B \left( x_0,
    \frac{1}{\varepsilon_0} \right)$ slides on the line $d$ to the right,
  the subdomain $D_I$ shrinks to 0, which makes the $C^{1,\alpha}$
  modulus to become unbounded \cite[page 93]{liGradientEstimatesSolutions2000}. 
  Moreover, in some cases, it is possible that a cusp also appears at some points (point $A$ on the zoomed in figure above).
  }
\label{fig:2}
\end{figure}

\end{remark}

The following result follows from \cref{sec:lipschitz-estimate}, the De Giorgi-Nash-Moser
estimate and a change of variable.
\begin{proposition}[Interior Lipschitz Estimate II]
\label{sec:lipschitz-estimate-5}
Suppose that $\bfa \in \fraM_{\per}(\lambda,\Lambda)$ and
$f \in L^{\infty}(\Omega)$. Fix
$x_0\in\Omega$, $R > 0$ so that $B(x_0,R)\subset \Omega$. There exist
$\varepsilon_0=\varepsilon_0 (\lambda,\Lambda, d)> 0$ and $C=C(\lambda,\Lambda, d) > 0$ such that, for all
$0 < \varepsilon < \varepsilon_0$, the
weak solution $\varphi^{\varepsilon} \in H^1(B(x_0,R))$ of the
equation
$-\Div \left[ \bfa \left( \frac{x}{\varepsilon} \right)\nabla
  \varphi^{\varepsilon} \right] = f$ in $B(x_0,R)$ satisfies:
\begin{align}
\label{eq:65}
  \abs{\nabla \varphi^{\varepsilon}(x_0)}
  \le C' \left[
  \left( \dashint_{B(x_0,R)} \abs{\nabla \varphi^{\varepsilon}(z)}^2  \di z
  \right)^{\frac{1}{2}} + R \sup_{x \in B(x_0,R)} \abs{f(x)}
  \right].
\end{align}
\end{proposition}

\begin{proof}
  Without loss of generality, we assume $x_0 = 0$. By
  \cref{sec:lipschitz-estimate}, with $R = 1$ and considering $\varphi^{\varepsilon} - \dashint_{B(0,1)}
  \varphi^{\varepsilon} (z) \di z $, which solves $-\Div \left[ \bfa \left( \frac{x}{\varepsilon} \right)\nabla
    \varphi^{\varepsilon} \right] = f$ in $B(0,1)$, we have that there exist
  $\varepsilon_0 > 0$ and $C' > 0$, depending only on $\lambda,\Lambda,d$, such that:
\begin{align}
  \label{eq:66}
  \begin{split}
  \abs{\nabla \varphi^{\varepsilon}(0)}
  &\le \norm{\nabla \varphi^{\varepsilon}}_{L^{\infty}(B(0,1/4))}\\
  &\le C' \left( \norm{\varphi^{\varepsilon} - \dashint_{B(0,1)}
  \varphi^{\varepsilon}(z) \di z}_{L^{\infty}(B(0,1/2))} +
\norm{f}_{L^{\infty}(B(0,1/2))} \right)\\
&\le C' \left(
  \norm{\varphi^{\varepsilon} - \dashint_{B(0,1)}
  \varphi^{\varepsilon}(z) \di z}_{L^{2}(B(0,1))} +
\norm{f}_{L^{\infty}(B(0,1))}
\right)\\
&\le C' \left(
  \norm{\nabla \varphi^{\varepsilon}}_{L^2(B(0,1))} + \norm{f}_{L^{\infty}(B(0,1))}
\right)\\
&\le C' \left[
  \left( \dashint_{B(0,1)} \abs{\nabla \varphi^{\varepsilon}(z)}^2  \di z
  \right)^{\frac{1}{2}} +  \sup_{x \in B(0,1)} \abs{f(x)}
  \right],
\end{split}
\end{align}
where we have used the De Giorgi-Nash-Moser estimate and Poincar\'{e}'s 
inequality.

For $R > 0$ and $x \in B(0,1)$, let $v^{\varepsilon}(x)\coloneqq R^{-1}
\varphi^{\varepsilon}(Rx),$ then $\nabla v^{\varepsilon}(x) = \nabla
\varphi^{\varepsilon}(Rx)$ and: 
\begin{align*}
-\Div \left[ \bfb \left( \frac{x}{\varepsilon} \right) \nabla
  v^{\varepsilon}(x) \right] = R f(Rx), 
\end{align*}
where $\bfb (z) \coloneqq \bfa (Rz).$ We have $\bfb \in \fraM
(\lambda,\Lambda)$ is $R^{-1}Y-$periodic. Note that the proof of
\cref{sec:lipschitz-estimate} does not depend on the period, hence, \eqref{eq:66} holds for $v^{\varepsilon}$ in
particular: 
\begin{align}
  \label{eq:67}
  \begin{split}
\abs{\nabla \varphi^{\varepsilon}(0)} &= \abs{\nabla v^{\varepsilon}(0)}\\
  &\le C' \left[
  \left( \dashint_{B(0,1)} \abs{\nabla v^{\varepsilon} (x)}^2  \di x
  \right)^{\frac{1}{2}} +  R \sup_{x \in B(0,1)} \abs{f(Rx)}
\right]\\
&= C' \left[
  \left( \dashint_{B(0,1)} \abs{\nabla \varphi^{\varepsilon}(Rx)}^2  \di x
  \right)^{\frac{1}{2}} +  R \sup_{x \in B(0,1)} \abs{f(Rx)}
\right].
\end{split}
\end{align}
By a change of variable in \eqref{eq:67}, we obtain \eqref{eq:65}.
\end{proof}

We recall the interior H\"{o}lder estimate, adapted from
\cite[Proposition 1]{prangeUniformEstimatesHomogenization2014} (or \cite[Lemma 9]{avellanedaCompactnessMethodsTheory1987}) that will be used to obtain the boundary H\"{o}lder estimate in
the next section. 

\begin{proposition}[Interior H\"{o}lder Estimate] 
\label{sec:lipschitz-estimate-3}
Suppose that $\bfa \in \fraM_{\per}(\lambda,\Lambda)$ and
$f \in L^{d + \gamma}(\Omega)$, for some $ \gamma >0$. Fix
$x_0\in\Omega$, $R > 0$ such that $B(x_0,R)\subset \Omega$. Let 
$0 < \mu \coloneqq \frac{\gamma}{d+\gamma} < 1$. There exists
 $C = C(\gamma,\lambda,\Lambda, p, R, d)> 0$ such that, for all
$\varepsilon > 0$, the
weak solution $\varphi^{\varepsilon} \in H^1(B(x_0,R))$ of the equation
$-\Div \left[ \bfa \left( \frac{x}{\varepsilon} \right)\nabla
  \varphi^{\varepsilon} \right] = f$ in $B(x_0,R)$ satisfies:
\begin{align}
\label{eq:41}
  \left[ \varphi^{\varepsilon}\right]_{C^{0,\mu}(B(x_0,R/2))}
  \le C  \left( \norm{\varphi^{\varepsilon}}_{L^2(B(x_0,R))} + \norm{f}_{L^{d+\gamma}(B(x_0,R))} \right),
\end{align}
where $\left[ h\right]_{C^{0,\mu} (A)} 
\coloneqq \sup_{x \ne y \in A} \frac{\abs{h(x)-h(y)}}{\abs{x-y}^\mu}$. 
\end{proposition}

The proof of \cref{sec:lipschitz-estimate-3} is similar to \cite[Proposition 1]{prangeUniformEstimatesHomogenization2014}. Indeed, a closer look at the proof of \cite[Proposition
1]{prangeUniformEstimatesHomogenization2014} reveals that the
H\"{o}lder continuity assumption on $\bfa$ is needed only for the classical Schauder estimate for elliptic systems to hold. However, this paper is devoted to the scalar case, and we use the De Giorgi-Nash-Moser Theorem \cite[Theorem
8.24]{gilbargEllipticPartialDifferential2001}, for which the assumption $\bfa$ is bounded is sufficient.

\begin{remark}
For the case of elliptic systems, the De Giorgi-Nash-Moser
Theorem does not hold, see counterexamples by De Giorgi, Giusti and Miranda, and others, c.f. \cite[Section
9.1]{giaquintaIntroductionRegularityTheory2012}, and references cited therein. Because of that, this paper is concerned with the scalar case only.
\end{remark}

\section{Boundary estimates, Green functions, Dirichlet correctors and  proof of main theorem}
\label{sec:bound-hold-estim}

The following result is adapted from \cite[Section
2.3]{avellanedaCompactnessMethodsTheory1987} and \cite[Section 5.2]{shenPeriodicHomogenizationElliptic2018}.
\begin{proposition}[Boundary H\"{o}lder Estimate]
\label{sec:bound-hold-estim-1}
Suppose that $\bfa \in \fraM_{\per}(\lambda,\Lambda)$,
and $\Omega$ is a $C^1$-domain. Fix $x_0\in \partial\Omega$,
$0 < r < \diam (\Omega)$ and
$0 < \mu  <1$. Let
$g \in C^{0,1}\left( B(x_0,r) \cap \partial \Omega \right).$ 
There exist $\varepsilon_0 =\varepsilon_0 (\mu,\lambda,\Lambda,d) > 0$ and $C=c(\mu,\lambda,\Lambda,d) > 0$ such that, for all $0 < \varepsilon < \varepsilon_0$, every
weak solution $\varphi^{\varepsilon} \in H^1(B(x_0,r))$ of the equation:
\begin{align*}
-\Div \left[ \bfa \left( \frac{x}{\varepsilon} \right)\nabla
  \varphi^{\varepsilon} \right]
  &= 0 \text{ in } B(x_0,r) \cap \Omega, \\
  \varphi^{\varepsilon}
  &= g \text{ on } B(x_0,r) \cap \partial \Omega,
\end{align*}
satisfies:
\begin{align}
  \label{eq:68}
  \begin{split}
  &\left[ \varphi^{\varepsilon}\right]_{C^{0,\mu}(B(x_0,r/2) \cap \Omega)}\\
  &\quad\le C r^{-\mu} \left[
  \left( \dashint_{B(x_0,r) \cap \Omega} \abs{\varphi^{\varepsilon} (z)}^2  \di z
  \right)^{\frac{1}{2}} + \abs{g(x_0)} + r
  \norm{g}_{C^{0,1}\left(B(x_0,r)\cap \partial \Omega\right)}
\right].
\end{split}
\end{align}

\end{proposition}

The proof of \cref{sec:bound-hold-estim-1} follows the proof of
\cite[Theorem 5.2.1]{shenPeriodicHomogenizationElliptic2018} with
minor modifications. In \cite[Theorem 5.2.1]{shenPeriodicHomogenizationElliptic2018}, the assumption that
$\bfa \in \mathrm{VMO} (\RR^d)$ is used only in two places: 1) when
$\varepsilon \ge \varepsilon_0$, which is beyond of the scope of this
particular theorem, and 2) to obtain the interior H\"{o}lder estimate,
which we already relaxed in \cref{sec:lipschitz-estimate-3}.

Thanks to \cref{sec:lipschitz-estimate-3} and
\cref{sec:bound-hold-estim-1}, we can drop the assumption that $\bfa
\in \mathrm{VMO} (\RR^d)$ of Theorem 5.4.1-2 and Lemma 5.4.5 in
\cite{shenPeriodicHomogenizationElliptic2018}.  The results are summarized in the following proposition.

\begin{proposition}[Green's functions]
\label{sec:bound-hold-estim-2}
Suppose that $\bfa \in \fraM_{\per}(\lambda,\Lambda )$ and $\Omega$ is a
$C^1$-domain. Fix $0 < \mu,\sigma,\sigma_1 < 1$ and let
$\delta (x) \coloneqq \dist (x,\partial \Omega).$ Then, there exist
$\varepsilon_0=\varepsilon_0(\mu,\lambda,\Lambda,d)> 0$ and $C=C(\lambda,\Lambda,\sigma,\sigma_1,\Omega) > 0$ such that, for all $0
< \varepsilon < \varepsilon_0$, the Green's functions
$G^{\varepsilon}(x,y)$ exist and satisfy the following: 
\begin{subequations}
  \label{eq:70}
\begin{align}
  \label{eq:71}
  &\abs{G^{\varepsilon}(x,y)}
  \le
    \begin{cases}
      C \frac{1}{\abs{x-y}^{d-2}} &\text{ if } d \ge 3,\\
      C \left[ 1 + \ln \left(  \frac{r_0}{\abs{x-y}}\right)
      \right] &\text{ if } d = 2.
    \end{cases}\\
  \label{eq:72}
  &\abs{G^{\varepsilon}(x,y)}
  \le
    \begin{cases}
      \frac{C\delta (x)^{\sigma}}{\abs{x - y}^{d-2+\sigma}} &\text{ if
      } \delta(x) < \frac{1}{2}\abs{x-y},\\
      \frac{C \delta(y)^{\sigma_1}}{\abs{x-y}^{d-2+\sigma_1}} &\text{
        if } \delta(y) < \frac{1}{2}\abs{x-y},\\
      \frac{C \delta(x)^{\sigma}
        \delta(y)^{\sigma_1}}{\abs{x-y}^{d-2+\sigma+\sigma_1}} &\text{
      if } \delta(x) < \frac{1}{2}\abs{x-y} \text{ or } \delta(y) <
    \frac{1}{2} \abs{x-y},
  \end{cases}\\
  &\int_{\Omega} \abs{\nabla_y G^{\varepsilon} (x,y)}
  \delta(y)^{\sigma-1} \di y
  \le C \delta(x)^{\sigma},
\end{align}
\end{subequations}
where $x,y \in \Omega, x \ne y$ and $r_0 \coloneqq \diam (\Omega).$

As a consequence, for $0<c<1$ and $g \in C^{0,1}(\Omega)$, there exists $C = C(\lambda,\Lambda,d) > 0$ such that, for any $x_0
\in \partial \Omega$, for any $\varepsilon$ satisfying $c \varepsilon \le \min\{c\varepsilon_0,r\} \le r < r_0 \coloneqq \diam
(\Omega)$, and for any
 solution $\varphi^{\varepsilon}$ of the Dirichlet problem
$-\Div \left[ \bfa \left( \frac{x}{\varepsilon} \right)\nabla
  \varphi^{\varepsilon} \right] = 0$ in $\Omega$,
$\varphi^{\varepsilon} = g$ on $\partial \Omega$, the following estimate holds:
\begin{align}
\label{eq:69}
\left( \dashint_{B(x_0,r)\cap \Omega} \abs{\nabla
  \varphi^{\varepsilon}}^2 \right)^{\frac{1}{2}}
  \le C \left[
  \norm{\nabla g}_{L^{\infty}(\Omega)} + \varepsilon^{-1} \norm{g}_{L^{\infty}(\Omega)}
  \right].
\end{align}
 
\end{proposition}

We now define the boundary Dirichlet corrector: For $1 \le i \le d,$
let $\Phi^{i,\varepsilon}\in H^1(\Omega)$ be the solution of the problem:
\begin{align}
\label{eq:36}
  \begin{split}
    -\Div \left[ \bfa \left( \frac{x}{\varepsilon}\right) \nabla
      \Phi^{i,\varepsilon}(x)  \right]
    &=0 \text{ for } x \in \Omega,\\
    \Phi^{i,\varepsilon}(z)
    &=z_i \text{ for } z \in \partial\Omega.
  \end{split}
\end{align}
The following proposition provides a bound on the boundary Dirichlet corrector.
\begin{proposition}
\label{sec:lipschitz-estimate-2}
Let $\Omega$ be a bounded $C^{1,\alpha}$-domain. Suppose that
$\bfa \in \fraM_{\per}(\lambda,L)$ is piecewise
$C^{\alpha}$-continuous. Then, for all $\varepsilon > 0$, the solution $\Phi^{i,\varepsilon}$ of
\eqref{eq:36} satisfies:
\begin{align}
\label{eq:37}
\norm{\nabla \Phi^{i,\varepsilon}}_{L^{\infty}(\Omega)} \le C,
\end{align}
where constant $C$ depends only on $\lambda, \Lambda$ and $\Omega$. 
\end{proposition}
The proof is similar to \cite[Theorem
5.4.4]{shenPeriodicHomogenizationElliptic2018}. One only needs to use
three observations: 
\begin{itemize}
\item The case $c\varepsilon \ge \min \{c\varepsilon_0, r\}$ follows
  from \cite[Theorem 1.2]{liGradientEstimatesSolutions2000}, by the
  same argument used in \cref{sec:impr-grad-estim}.
\item Let $\vec{\omega} = (\omega^1,\omega^2,\ldots,\omega^d)$ be the
  solutions of the cell problems \eqref{eq:pp453}. Given only that $\bfa$ is {piecewise}
  $C^{\alpha}$-continuous, then $\nabla \vec{\omega}$ is bounded in $L^{\infty}(Y, \RR^{d\times d})$,
  see the first paragraph in the proof of \cite[Theorem
  3.2]{francfortEnhancementElastodielectricsHomogenization2021} or
  \cite[Corollary 3.5]{zhangInteriorHolderGradient2013a}.
\item The interior Lipschitz estimate in
  \cref{sec:lipschitz-estimate-5} only requires $\bfa$ is piecewise H\"{o}lder continuous.
\end{itemize}
We now combine \cref{sec:lipschitz-estimate-2},
\cref{sec:lipschitz-estimate-5} and \cite[Theorem
1.2]{liGradientEstimatesSolutions2000} to obtain a discontinuous
coefficient-version of \cite[Theorem
5.5.1]{shenPeriodicHomogenizationElliptic2018}.
\begin{proposition}[Boundary Lipschitz Estimate]
\label{sec:bound-hold-estim-3}
Suppose that $\bfa \in \fraM_{\per}(\lambda,L)$ is piecewise
$C^{\alpha}$-continuous,
and $\Omega$ is a $C^{1,\alpha}$-domain. Fix $x_0\in \partial\Omega$,
$0 < r < \diam (\Omega)$ and
$0 < \mu  <1$. Let
$g \in C^{1,\alpha}\left( B(x_0,r) \cap \partial \Omega \right).$ There exist $\varepsilon_0 =\varepsilon_0 (\mu,\lambda,\Lambda,d,\Omega) > 0$ and $C=C(\mu,\lambda,\Lambda,d,\Omega) > 0$ such that, for all $0 < \varepsilon < \varepsilon_0$, the
weak solution $\varphi^{\varepsilon} \in H^1(B(x_0,r))$ of the equation:
\begin{align*}
-\Div \left[ \bfa \left( \frac{x}{\varepsilon} \right)\nabla
  \varphi^{\varepsilon} \right]
  &= 0 \text{ in } B(x_0,r) \cap \Omega, \\
  \varphi^{\varepsilon}
  &= g \text{ on } B(x_0,r) \cap \partial \Omega,
\end{align*}
satisfies:
\begin{align}
  \label{eq:73}
  \begin{split}
  &\norm{\nabla \varphi^{\varepsilon}}_{L^{\infty}(B(x_0,r/2)\cap
  \partial \Omega)}\\
  &\le C \left[
    r^{-1} \left( \dashint_{B(x_0,r)\cap \Omega}
    \abs{\varphi^{\varepsilon}}^2 \right)^{\frac{1}{2}} + r^{\alpha}
    \norm{\nabla_{\tan}g}_{C^{0,\alpha}(B(x_0,r) \cap \partial
    \Omega)} \right.\\
  &\quad \qquad \left.
 {}{+}{}   \norm{\nabla_{\tan}g}_{L^{\infty}(B(x_0,r) \cap \partial \Omega)}
    + r^{-1} \norm{g}_{L^{\infty}(B(x_0,r) \cap \partial \Omega)} \vphantom{ \left( \dashint_{B(x_0,r)\cap \Omega}
    \abs{\varphi^{\varepsilon}}^2 \right)^{\frac{1}{2}}}
  \right].
  \end{split}
\end{align}

\end{proposition}

The estimate \eqref{eq:34} of \cref{sec:gradient-estimate-1} is a consequence of \cref {sec:lipschitz-estimate-5} and
  \cref{sec:bound-hold-estim-3}, by an argument similar to  \cite[Theorem 5.6.2]{shenPeriodicHomogenizationElliptic2018}.

\section{Application to magnetic suspensions}

\label{sec:appl-magn-susp}

In this section, we apply the regularity results obtained above to the rigorous homogenization procedure discussed in \cite{dangHomogenizationNondiluteSuspension2021}. For that, we first recap the formulation of the fine-scale problem and the homogenization result itself. We begin by introducing the definition of two-scale convergence, which will be used below.
\begin{definition}
A sequence $\{ v^\varepsilon \}_{\varepsilon>0}$ in $L^2(\Omega)$ is said to \emph{two-scale converge} to $v = v({x},{y})$, with $v \in L^2 (\Omega \times Y)$, if and only if:
\begin{align*}
    \lim_{\varepsilon \to 0} \int_\Omega v^\varepsilon(\xx) \psi \left( \xx, \frac{\xx}{\varepsilon}\right) \di \xx 
    = \frac{1}{\abs{Y}} \int_\Omega \int_Y v(\xx,\yy) \psi(\xx,\yy) \di \yy \di \xx,
\end{align*}
for any test function $\psi = \psi ({x}, {y})$, with $\psi \in \mathcal{D}(\Omega, C_\per^\infty (Y))$,
see \cite{allaireHomogenizationTwoscaleConvergence1992,cioranescuIntroductionHomogenization1999,nguetsengGeneralConvergenceResult1989}. In this case, we write $v^\varepsilon \tscale v$.
\end{definition}

Let the kinematic viscosity be denoted by $\nu = \frac{\eta}{\rho_f}$, where $\eta>0$ and $\rho_f>0$ are the fluid viscosity and the fluid density, respectively.  The dimensionless quantities that appear in this problem are the (hydrodynamic) \emph{Reynolds number}
$\nRe = UL/\nu$ , the \emph{Froude number} $\nFr =U/\sqrt{FL}$ and the \emph{coupling parameter}
$S = \frac{B^2}{\rho_f \Lambda U^2},$
where $L, U, B$ and $F$ are the characteristic scales
corresponding to length, fluid velocity,
magnetic field and body density force, respectively. Moreover, $ \Lambda >0$ is defined in \ref{cond:a-bound}.

From now on, we suppose $\Omega$ is $C^{3,\alpha}$, which is
needed for the corrector result below. Suppose further that
$\bg \in H^1({\Omega},\RR^d)$, $k \in C^{1,\alpha}(\partial {\Omega})$
and $f \in L^{\infty}(\Omega)$.  Let $\uu^{\varepsilon}$ and
$p^{\varepsilon}$ be the fluid velocity and the fluid pressure,
respectively. Also, in a space free of current, the magnetic field
strength is given by
$\mathbf{H}^{\varepsilon} = \nabla \varphi^{\varepsilon}$, for some
magnetic potential $\varphi^{\varepsilon}(\xx)$.  Let
$\uu^{\varepsilon} \in H_0^1(\Omega,\RR^d)$,
$p^{\varepsilon}\in L^2(\Omega)/\RR$, and
$\varphi^{\varepsilon} \in H^1(\Omega)$ be the solution of the
following boundary value problem:
\begin{subequations}
  \label{eq:p410} 
\begin{align}
\label{eq:p411}
  -\Div \left[ \vec{\sigma}(\uu^{\varepsilon},p^{\varepsilon}) +
  \bT(\varphi^{\varepsilon}) \right]
  &= \frac{1}{\nFr^2} \bg, && \text{ in } \Omega_f^{\varepsilon}\\
  \label{eq:p419}
  \Div \uu^{\varepsilon}
  &= 0, && \text{ in } \Omega_f^{\varepsilon}\\
  \label{eq:p420}
\DD(\uu^{\varepsilon}) 
  &= 0, &&  \text{ in }\Omega_s^{\varepsilon} \\
  \label{eq:p421}
  -\Div \left[ \bfa \left( \frac{\xx}{\varepsilon} \right) \nabla
  \varphi^{\varepsilon} \right]
  &=  f && \text{ in }\Omega,
\end{align}
\end{subequations}
together with the balance equations:
\begin{subequations}
  \label{eq:p413}
\begin{align}
\label{eq:p414}
\int_{\Gamma_{i}^{\varepsilon}} \left[
  \vec{\sigma}(\uu^{\varepsilon},p^{\varepsilon}) +
  \bT(\varphi^{\varepsilon}) \right] \nn_i\di \hmeas
  &= 0,\\
  \label{eq:p422}
  \int_{\Gamma_{i}^{\varepsilon}} \left( \left[
  \vec{\sigma}(\uu^{\varepsilon},p^{\varepsilon}) +
  \bT(\varphi^{\varepsilon}) \right]\nn_i \right) \times \nn_i\di \hmeas
  &= 0,
\end{align}
\end{subequations}
and boundary conditions:
\begin{subequations}
  \label{eq:p415} 
\begin{align}
\label{eq:p416}
  \uu^{\varepsilon}
  &= 0, \text{ on }\partial \Omega,\\
  \label{eq:p423}
 \varphi^{\varepsilon}
  &= k
  , \text{ on }\partial \Omega,
\end{align}
\end{subequations}
where:
\begin{subequations}
  \label{eq:p417} 
\begin{align}
\label{eq:p418}
  \vec{\sigma}(\uu^{\varepsilon},p^{\varepsilon})
  &\coloneqq \frac{2}{\nRe} \DD(\uu^{\varepsilon}) -
    p^{\varepsilon}\II,\\  \DD(\uu^{\varepsilon}) 
&\coloneqq \frac{\nabla \uu^{\varepsilon} + \nabla^{\top}\uu^{\varepsilon}}{2},\\
  \label{eq:p426}
  \bT(\varphi^{\varepsilon})
  &\coloneqq S \bfa \left( \frac{\xx}{\varepsilon} \right) \left( \nabla
    \varphi^{\varepsilon} \otimes \nabla \varphi^{\varepsilon} -
    \frac{1}{2} \abs{\nabla \varphi^{\varepsilon}}^2 \II \right).
\end{align}
\end{subequations}
are the \textit{rate of strain}, the \textit{Cauchy stress} and the \textit{Maxwell stress} tensors, respectively. 
For the detailed derivation and the physical meaning of the equations above, we refer the readers to
\cite{dangHomogenizationNondiluteSuspension2021} and references
therein.  Observe that, in the context of this paper, we consider the Dirichlet boundary condition
\eqref{eq:p423}, instead of  a Neumann boundary condition \eqref{NC} in \cite{dangHomogenizationNondiluteSuspension2021}, to relax the regularity
assumption on the magnetic permeability needed in \cite{dangHomogenizationNondiluteSuspension2021}.
Then, the weak formulation for \eqref{eq:p421} and \eqref{eq:p423} is given by: 
\begin{align}
  \label{eq:21}
  \begin{split}
\int_{\Omega} &\bfa \left( \frac{x}{\varepsilon} \right) \nabla
  \left(\varphi^{\varepsilon} - k  \right) \cdot \nabla \xi \di x\\
  &\quad= -\int_{\Omega} \bfa \left( \frac{x}{\varepsilon} \right) \nabla
    k  \cdot \nabla \xi \di x +
    \int_{\Omega} f \xi \di x, \quad \forall \xi \in H_0^1(\Omega).
    \end{split}
\end{align}
One immediately has that 
$\norm{\varphi^{\varepsilon}}_{H^1(\Omega)} \le C \left(
  \norm{k}_{H^{1/2}(\partial \Omega)} +
  \norm{f}_{L^q(\Omega)}\right)$, which implies that
$\varphi^{\varepsilon}$ is two-scale convergent. Choosing a test
function as in \cite[Lemma
3.7]{dangHomogenizationNondiluteSuspension2021}, we obtain the cell
problem \eqref{eq:p453} and the first two effective equations defined in
\eqref{summary-eqn} below.

Moreover, \cref{sec:gradient-estimate-1} ensures that $\nabla
\varphi^{\varepsilon}$ is uniformly bounded in $L^\infty(\Omega,
\RR^d)$, with respect to
$\varepsilon \in (0,\varepsilon_0)$.  Therefore, we obtain the existence, uniqueness and a priori
bounds for $\uu^{\varepsilon}$ and $p^{\varepsilon}$ as in
\cite[Corollary 3.11]{dangHomogenizationNondiluteSuspension2021}. Here, we have relaxed the restrictive assumption \eqref{supinf} made in \cite{dangHomogenizationNondiluteSuspension2021} and we can use our results in the case when the
constant magnetic permeability is anisotropic, namely, when $\bfa$ is a matrix.

To carry on with the homogenization formulation, for $1 \le i,j \le d,$ denote by $\bP^{ij}$ the vector defined by $\bP^{ij}_k \coloneqq y_j
\delta_{ik}$.  Consider $\omega^i \in H^1_{\per}(Y)/\RR$, the solution of: 
  \begin{equation}
\label{eq:p453}
-\Div_{\yy} \left[ 
{\bfa(\yy) }\left(\ee^i +\nabla_{\yy} \omega^i(\yy)
  \right)  \right]
  = 0 \text{ in } Y.
  \end{equation}
Also, consider $\vec{\chi}^{ij} \in H^1_{\per}(Y,\RR^d)/\RR$ and
$q^{ij}\in L^2(Y)/\RR$, solving:
\begin{align}
\label{eq:504}
  \begin{split}
    \Div_{\yy} \left[ \DD_{\yy} \left(\bP^{ij} - \vec{\chi}^{ij}
      \right) + q^{ij}\II \right]
    &= 0 \text{ in } Y_f,\\
    \Div_{\yy} \vec{\chi}^{ij}
    &= 0 \text{ in }Y,\\
    \DD_{\yy} \left( \bP^{ij} - \vec{\chi}^{ij} \right)
    &= 0 \text{ in } Y_s,\\
    \int_{\Gamma}  \left[
      \DD_{\yy} \left( \bP^{ij} - \vec{\chi}^{ij}
      \right) - q^{ij}\II
    \right] \nn_{\Gamma} \di \hmeas 
    &=0,\\
    \int_{\Gamma}  \left[
       \DD_{\yy} \left( \bP^{ij} - \vec{\chi}^{ij}
      \right) - q^{ij}\II
    \right] \nn_{\Gamma} \times \nn_{\Gamma} \di \hmeas 
    &=0,
    \end{split}
    \end{align}
and consider $\vec{\xi}^{ij} \in H^1_{\per}(Y,\RR^d)/\RR$ and
$r^{ij} \in L^2(Y)/\RR$, solving:
\begin{align}
\label{eq:504b}
  \begin{split}
    \Div_{\yy} \left[  \DD_{\yy} \left( \vec{\xi}^{ij}
      \right) + r^{ij}\II +  \vec{\tau}^{ij} \right]
    &= 0 \text{ in } Y_f,\\
    \Div_{\yy} \vec{\xi}^{ij}
    &= 0 \text{ in }Y,\\
    \DD_{\yy} \left( \vec{\xi}^{ij} \right)
    &= 0 \text{ in } Y_s,\\
    \int_{\Gamma}  \left[
      \DD_{\yy} \left( \vec{\xi}^{ij}
      \right) + r^{ij}\II +  \vec{\tau}^{ij}
    \right] \nn_{\Gamma} \di \hmeas  
    &=0,\\
    \int_{\Gamma}  \left[
      \DD_{\yy} \left( \vec{\xi}^{ij}
      \right) + r^{ij}\II +  \vec{\tau}^{ij}
    \right] \nn_{\Gamma} \times \nn_{\Gamma} \di \hmeas 
    &=0.
  \end{split}
\end{align}

We also define:
\begin{align}
\label{eq:513}
  \begin{split}
    \calA_{jk}
  &\coloneqq \frac{1}{\abs{Y}} \int_Y \bfa(\yy) (\ee^k+ \nabla \omega^k(\yy))\cdot
  (\ee^j + \nabla \omega^j(\yy)) \di \yy,\\
  \calN^{ij}_{mn}
  &\coloneqq \frac{1}{\abs{Y}} \int_Y 
  \DD_{\yy}(\bP^{ij}-\vec{\chi}^{ij}) : \DD_{\yy}(\bP^{mn}-\vec{\chi}^{mn})
  \di \yy,\\
  \vec{\tau}_{\cel}^{ij}
  &\coloneqq
   { \bfa (\yy)} \left[ (\ee^i + \nabla_{\yy} \omega^i)  \otimes ( \ee^j
    + \nabla_{\yy} \omega^j) - \frac{1}{2} ( \ee^i +
    \nabla_{\yy} \omega^i ) \cdot ( \ee^j + \nabla_{\yy}
    \omega^j ) \II \right],~ y \in Y,\\
  \mathcal{B}^{ij}
  &\coloneqq
    \frac{1}{\abs{Y}} \int_Y \left( \DD_{\yy}(\vec{\xi}^{ij}) 
    +  \vec{\tau}^{ij}\right)\di \yy,
    \end{split}
\end{align}
where $\calA$ is the \emph{effective magnetic permeability},
which is symmetric and elliptic. The tensor $\calN \coloneqq \left\{ \calN^{ij}_{mn}\right\}_{1 \le i,j,m,n \le
  d}$ is the \emph{effective viscosity}, and it is a fourth rank
tensor.  Moreover, $\calN$ is symmetric, i.e.
$\calN^{ij}_{mn} = \calN^{mn}_{ij} = \calN^{ji}_{mn} =
\calN^{ij}_{nm}$, and it satisfies the Legendre-Hadamard
condition (or strong ellipticity condition), i.e. there exist
$\beta > 0$ such that, for all ${\zeta}, \eta \in \RR^d$,
one has $\calN^{ij}_{mn} \zeta_i \zeta_m \eta_j \eta_n \ge \beta \abs{{\zeta}}^2 \abs{{\eta}}^2$.
The matrix $\vec{\tau}_{\cel}$ is the \emph{Maxwell stress tensor} on $Y$, and $\calB$ is the \emph{effective coupling matrix}.

By the same argument as in Theorem~3.5, Lemma~3.9 and Lemma~3.14 of \cite{dangHomogenizationNondiluteSuspension2021}, the following result holds:
\begin{theorem} \label{cor:main}
Let $(\varphi^{\varepsilon}, \uu^{\varepsilon},p^{\varepsilon}) \in H^1(\Omega) \times H_0^1(\Omega,\RR^d) \times L_0^2(\Omega)$ be the solution of \eqref{eq:p410}. Then
\begin{align*}
    \varphi^\varepsilon 
    &\rightharpoonup \varphi^0 \text{ in } H^1(\Omega), \\
    \uu^\varepsilon
    &\rightharpoonup \uu^0 \text{ in } H_0^1(\Omega,\RR^d),\\
    p^\varepsilon 
    &\rightharpoonup \pi^0 \text{ in } L_0^2(\Omega),
\end{align*}
where $\varphi^0, \uu^0$ and $\pi^0$ are solutions of: 
\begin{align}
\label{summary-eqn}
\begin{aligned}
    -\Div \left( \calA \nabla\varphi^0 \right) 
    &= f &&\text{ in }\Omega,\\
    \varphi^0 
    &=k &&\text{ on }\partial\Omega,\\
    \Div \left[ 
  \frac{2}{\nRe}\calN^{ij}\DD \left( \uu^0 \right)_{ij} - {\pi^0} + 
  S\mathcal{B}^{ij} \frac{\partial \varphi^0}{\partial x_i}
  \frac{\partial \varphi^0}{\partial x_j}\right]
  &= \frac{1}{\nFr^2} \bg &&\text{ in }\Omega,\\
  \Div \uu^0 
  &= 0 &&\text{ in }\Omega\\
  \uu^0
  &= 0 && \text{ on } \partial \Omega,
  \end{aligned}
\end{align}
with $\calA,\calN^{ij},$ $\mathcal{B}^{ij}$, $1 \le i,j \le d$, defined in \eqref{eq:513}.
Moreover, the first-order correctors satisfy:
\begin{align*}
\lim_{\varepsilon\to 0}\norm{\nabla \varphi^{\varepsilon} (\cdot) - \nabla \varphi^0(\cdot) -
  \nabla_{\yy} \varphi^1 \left( \cdot, \frac{\cdot}{\varepsilon}
  \right)}_{L^2(\Omega,\RR^d)} &=0,\\
  \lim_{\varepsilon \to 0} \norm{\DD(\uu^{\varepsilon})(\cdot) - \DD
  (\uu^0)(\cdot) - \DD_{\yy} (\uu^1) \left( \cdot,
  \frac{\cdot}{\varepsilon} \right)}_{L^2(\Omega,\RR^{d\times d})}
  &= 0,
\end{align*}
where: 
\begin{align*}
 \varphi^1(\xx,\yy)
  &\coloneqq \omega^i(\yy) \frac{\partial \varphi^0}{\partial x_i}(\xx),\\
  \uu^1(\xx, \yy)
  &\coloneqq -\DD\left(\uu^0(\xx)\right)_{ij} \vec{\chi}^{ij}(\yy) +
  S\frac{\partial \varphi^0}{\partial x_i}(\xx) \frac{\partial
  \varphi^0}{\partial x_j}(\xx) \vec{\xi}^{ij}(\yy).
\end{align*}
\end{theorem}

\section{Conclusions}
\label{sec:conclusions}
This paper concerns a homogenized description of a non-dilute
suspension of magnetic particles in a viscous flow. The results demonstrated
in this paper generalize the ones obtained by the authors in
\cite{dangHomogenizationNondiluteSuspension2021}, where a more
restrictive assumption on the magnetic permeability \eqref{supinf} was used and a
Neumann boundary condition \eqref{NC} was imposed instead of the Dirichlet condition
\eqref{DC}. 
\cref{cor:main} above demonstrates the
\emph{effective response} of a viscous fluid with a locally periodic
array of paramagnetic/diamagnetic particles suspended in it, given by
the system of equations \eqref{eq:p410}. The effective equations are
described by \eqref{summary-eqn}, with the effective coefficients given
in \eqref{eq:513}.  These effective quantities depend only on the
instantaneous position of the particles, their geometry, and the
magnetic and flow properties of the original suspension described by
\eqref{eq:p410}.
Using the tools introduced in
  \cite{liGradientEstimatesSolutions2000} and the compactness method, an
  improved regularity estimate for the gradient of the magnetic potential of the
  original fine-scale problem \eqref{eq:p410} was obtained, see
  \cref{sec:gradient-estimate-1}. This theorem allows us to drop the restrictive assumption \eqref{supinf} mentioned above. Comparing to the classical results on regularity of this type, we do require the coefficient matrix
  belongs to a VMO-space, see e.g. \cite{avellanedaCompactnessMethodsTheory1987, prangeUniformEstimatesHomogenization2014, shenPeriodicHomogenizationElliptic2018}. Recently, in \cite[Proposition
  3.1]{francfortEnhancementElastodielectricsHomogenization2021}, the authors obtained an $L^q$-bound of the
  gradient of the solution of the scalar divergence equation, uniform with respect to $\varepsilon$, for $ q < \infty$. Our result, in \cref{sec:gradient-estimate-1}, shows that the gradient bound actually holds for the case
  $q = \infty$.

\section*{Acknowledgements}
The authors would like to thank Antoine Gloria for bringing up the
results in
\cite{francfortEnhancementElastodielectricsHomogenization2021} to
their attention. The first author would also like to thank Christophe
Prange for allowing him to audit his class, in which he first learned about the compactness method used in this paper.  
The work of the first author was partially supported by the NSF through grant DMS-1350248.
The work of the third author was supported  by NSF grant DMS-2110036.  This
material is based upon work supported by and while serving at the
National Science Foundation for the second author Yuliya Gorb. Any
opinion, findings, and conclusions or recommendations expressed in
this material are those of the authors and do not necessarily reflect
views of the National Science Foundation.

\appendix

\section{Appendix}
\label{sec:an-appendix}

\begin{theorem}[Interior Schauder Estimates {\cite{gilbargEllipticPartialDifferential2001,hanEllipticPartialDifferential2011}}] 
\label{sec:an-appendix-1}
Let $\bfb \in \fraM (\lambda,\Lambda)$ be a constant matrix and $w \in
H^1(\Omega)$ be a weak solution of: 
\begin{align*}
\bfb_{ij} D_i D_jw =  f+ \sum_{i=1}^d D_i f_i.
\end{align*}
For every $\alpha \in (0,1)$, there exists a uniform constant
$C=C(\alpha,d,\lambda,\Lambda)$ such that if $\Omega' \subset\subset
\Omega$, with $\delta = \dist (\Omega',\partial\Omega)$, then the
following estimates hold: 
\begin{itemize}
\item[(i)] If $f \in L^p(\Omega)$, $f_i \in L^q(\Omega)$ and
  $\alpha = 1 - \frac{d}{q}= 2 - \frac{d}{p}\in (0,1)$, then $w \in
  C^{\alpha}(\Omega')$ and: 
\begin{align*}
  \norm{w}_{C^{\alpha}(\Omega')}
  \le C \delta^{-\frac{d}{2}+1-\alpha} \left(
  \norm{f}_{L^p(\Omega)} + \sum_{i=1}^d \norm{f_j}_{L^q(\Omega)} + \norm{w}_{H^1(\Omega)}
  \right).
\end{align*}
\item[(ii)] If $f \in L^p(\Omega),$ $\alpha = 1 - \frac{d}{p} \in (0,1)$ and
  $f_i \in C^{\alpha}(\Omega),$ then $\nabla w \in
  C^{\alpha}(\Omega')$ and: 
\begin{align*}
  \norm{\nabla w}_{C^{\alpha}(\Omega')}
  \le C \delta^{-\frac{d}{2}-\alpha} \left(
  \norm{f}_{L^p(\Omega)} + \sum_{i=1}^d
  \norm{f_i}_{C^{\alpha}(\Omega)} + \norm{w}_{H^1(\Omega)}
  \right).
\end{align*}
\end{itemize}

\end{theorem}

\bibliographystyle{spmpsci}
\bibliography{homogenisation}

\end{document}